\documentclass[12pt]{amsart}

\usepackage{amssymb,color,url}

\usepackage{enumerate}

\usepackage{graphicx}

\makeatletter
\@namedef{subjclassname@2010}{%
  \textup{2010} Mathematics Subject Classification}
\makeatother

\newtheorem{thm}{Theorem}[section]
\newtheorem{cor}[thm]{Corollary}
\newtheorem{lem}[thm]{Lemma}
\newtheorem{prop}[thm]{Proposition}

\theoremstyle{definition}

\numberwithin{equation}{section}

\newcommand{\Z}{\mathbb{Z}}
\newcommand{\Q}{\mathbb{Q}}

\newcommand{\gen}{\text{gen}}

\newcommand{\spn}{\text{spn}}

\newcommand{\lan}{\langle}
\newcommand{\ran}{\rangle}

\newcommand{\muu}{\hat{\mu}}
\newcommand{\DD}{\mathcal D}
\newcommand{\PP}{\mathcal P}
\newcommand{\btu}{\bigtriangleup}

\begin{document}


\baselineskip=17pt



\title[one-class spinor genera]{Classification of one-class spinor genera for quaternary quadratic forms }

\author[A. G. Earnest]{A. G. Earnest}
\address{Department of Mathematics\\
Southern Illinois University\\
Carbondale, IL 62901, USA}
\email{aearnest@siu.edu}

\author[A. Haensch]{Anna Haensch}
\address{Department of Mathematics and Computer Science\\ 
Duquesne University\\
Pittsburgh, PA 15282, USA.}
\email{haenscha@duq.edu}

\date{}

\begin{abstract}
A quadratic form has a one-class spinor genus if its spinor genus consists of a single equivalence class.  In this paper, we determine that there is, up to equivalence, only one primitive integral positive definite quaternary quadratic form which has a one-class spinor genus but not a one-class genus.  In all other cases, such quaternary forms either have a genus and spinor genus which coincide, or the genus splits into multiple spinor genera, which in turn split into multiple equivalence classes.  
\end{abstract}



\maketitle

\section{Introduction}

An integral quadratic form is said to have a {\em one-class (spinor) genus} if its (spinor) genus consists of a single equivalence class (that is, if the (spinor) genus of the form has (spinor) class number 1). Recent work of Kirschmer and Lorch \cite{KL}, which completes the determination of all one-class genera of positive definite primitive integral quadratic forms in at least three variables, brings us naturally to revisit the corresponding problem for one-class spinor genera of such forms; that is, the classification of forms whose spinor genus consists of a single equivalence class. Our goal here is to complete this determination in the final remaining case of quaternary forms. The main result will be explicitly stated in Theorem 1.1.  

\subsection{Previous Results}

For brevity, the term {\it form} will refer throughout this paper to a positive definite integral quadratic form. By a general result of Rehmann \cite[Satz 2]{R71}, it is known {\it a priori} that there exist only finitely many one-class spinor genera of primitive forms of rank at least 3. It was proven by Earnest and Hsia that for forms of rank at least 5, the notions of one-class genus and one-class spinor genus coincide \cite{EH91}.  However, when the rank is less than or equal to 4, there exist one-class spinor genera which lie in genera containing multiple classes.  

When the rank is equal to 3, there are 27 such forms appearing in Jagy's list of spinor regular ternary forms (that is, forms that represent all integers represented by their spinor genus) that are not regular \cite{J}. In light of the work of the present authors in \cite{EH17}, it is now known that this list is complete, in the sense that it contains representatives from all one-class spinor genera of primitive ternary forms that are not regular. A check of the forms of class number exceeding 1 which appear in the list of 913 regular ternary forms given in \cite{JKS} yields an additional 18 forms with spinor class number 1. There is thus a total of 45 one-class spinor genera of primitive ternary forms that have class number exceeding 1. For completeness, a list of representatives of these one-class spinor genera will be given at the end of this paper.

To complete the determination of all one-class spinor genera for forms in at least three variables, it thus remains to fully investigate the one-class spinor genera of quaternary forms. There is one example of a quaternary form which lies in a one-class spinor genus, but not a one-class genus, that has appeared several times in the literature. In his book \cite[p. 114]{W}, Watson notes that the spinor genus of the quaternary form
\begin{eqnarray}\label{example}
x^2 +xy+7y^2 +3z^2 +3zw+3w^2
\end{eqnarray}
of discriminant\footnote{By the {\em discriminant} of a form $f$, denoted {\em disc(f)}, we will mean the determinant of the matrix of second partial derivatives of the form.} $3^6 = 729$ contains only one class, but its genus contains more than one spinor genus. It can be checked that the genus of this form consists of two spinor genera and a total of three classes. In his book \cite{N} which contains tables of all quaternary quadratic forms of discriminant at most 1732, Nipp notes on p. 14 that the only discriminant in this range that could admit multiple spinor genera is 729, and he goes on to show that for all forms of discriminant 729 other than those equivalent to (\ref{example}) the genus and spinor genus coincide.

In the case of a quaternary form that is equivalent to the norm form on a quaternion order, of which the form (\ref{example}) is an example, there are interesting connections between one-class spinor genera and algebraic properties of the underlying order. Over general Dedekind domains in global fields, Nipp \cite{N75} gives a characterization of one-class spinor genera in terms of the ideal theory of the order. In the same paper, Nipp also shows that for ternary forms the one-class spinor genus property is equivalent to an ideal-theoretic property of an associated quaternion order. In the case of rational quaternion orders, Estes and Nipp \cite{EN89} give a characterization of the one-class spinor genus property in terms of factorization in the order, extending investigations of Pall and Williams \cite{P46}, \cite{WP45} who characterized the one-class genera of quaternion orders in terms of quaternion factorization and determined the 39 orders having this property. The form (\ref{example}) appears in both of the papers \cite{N75} and \cite{EN89}. In fact, Parks \cite{P74} subsequently proved that the lattice corresponding to this form is a representative of the only isomorphism class of definite quaternion orders in rational quaternion algebras that lies in a spinor genus of one class, but a genus consisting of multiple classes. The forms covered by that result are rather special; for example, they all have square discriminant and represent 1.

\subsection{Main Result}

In the present paper, we show that in fact the form (\ref{example}) is a representative of what is essentially the only equivalence class of quaternary forms that coincides with its spinor genus but not its genus.  More precisely we prove:

\begin{thm}\label{maintheorem}
Let $f$ be a primitive integral positive definite quaternary quadratic form for which the spinor genus and class coincide.  Then either the genus and class of $f$ coincide or $f$ is equivalent to the form (\ref{example}).
\end{thm}

\subsection{Method of Proof}

To prove this result, the general strategy is as follows. Using a transformation $\mu_p$ introduced by Gerstein in \cite{G72}, which is similar to the ``$p$-mapping" defined by Watson in \cite{W63}, we first associate to each one-class spinor genus form a one-class genus form whose discriminant has the same prime factors as the original form. By cross-referencing with the list of one-class genus quaternary forms appearing in the classification by Kirschmer and Lorch in \cite{KL}, we produce a small list of possible prime divisors for discriminants of one-class spinor genus forms. We then systematically eliminate candidate discriminants by using the $\mu_p$ transformations to show that the associated genus does not split into multiple spinor genera, and hence the form is not of relevance to us, or by using a version of the Minkowski-Siegel mass formula to show that the spinor genus must contain more than one class. For any cases that do not succumb to these methods, we generate all equivalence classes of forms of a targeted discriminant, and then explicitly compute the numbers of classes in the genus and spinor genus in each case.

This strategy makes use of three critical and  interconnected computational components.  The first is the online L-Functions and Modular Forms Database \cite{LMFDB}, in which the lattice database contains the full list of one-class genus forms determined by Kirschmer and Lorch \cite{KL}.  Downloading the list, the entries can be viewed as objects in the class of quadratic forms in Sagemath \cite{Sage}, enabling quick computations of discriminant and local structure of the forms.  From here, it can be easily determined what sort of local splittings and discriminant divisors are admissible by one-class genus forms.  In certain cases, we will need to generate a list of equivalence classes of forms of a fixed discriminant. This is done by means of an algorithm first given by Pall \cite{P46} and described in \cite[Lemma 3]{EN91}, in conjunction with the functionality for testing local and global isometry in Magma.  Once a list of potential candidates has been determined in these cases, local structures can be computed and compared against those already computed in Sagemath.  For all remaining candidates, Magma can be used to explicitly compute the number of classes in the genus and spinor genus.  

\subsection{Organization}

The remainder of this paper is organized as follows. In \S 2, relevant notation and terminology for lattices are summarized.  The $\mu_p$-transformations are reviewed in \S 3, some explicit local computations for the action of these transformations are given and a first list of possible prime divisors of the discriminants of one-class spinor genera is produced. The Conway-Sloane version of the mass formula is described in \S 4, bounds are computed for the factors in the mass formula for forms with one-class spinor genera, and bounds are obtained for the powers of certain primes that could possibly occur as discriminants of one-class spinor genera. In \S 5, the list of possible prime divisors of discriminants of one-class spinor genera will be systematically reduced, ultimately showing that only $2$ and $3$ can occur. Then, using the $\mu_p$-transformations, mass formula bounds, and explicit examination of forms of several targeted discriminants, in \S 6 all remaining candidates for one-class spinor genera are eliminated except for $3^6$, as claimed. Two appendices are given in \S 7. In the first, explicit computer code is given for the implementation of the algorithm of Pall used to generate representatives of the equivalence classes of quaternary forms of discriminant $Dp^2$ from those of discriminant $D$. Finally, \S 7.2 contains the list of one-class spinor genera of primitive ternary forms for which the genus and spinor genus do not coincide. 

\section{Preliminaries and Notation}\label{prelim}

For the remainder of this paper we will abandon the language of forms, and instead adopt the geometric language of lattices as favored in the recent literature on this topic, especially \cite{KL}. Any unexplained notation and terminology, as well as basic background results, can be found in \cite{OM}. To set the context, let $R$ be an integral domain with field of quotients $F$ of characteristic not $2$, and let $(V,Q)$ be a nondegenerate quadratic space over $F$ with associated symmetric bilinear form $B:V\times V\rightarrow F$ for which $B(v,v)=Q(v)$ for all $v\in V$. An {\em $R$-lattice on $V$} (or simply {\em $R$-lattice} if it is unnecessary to specify the underlying quadratic space) is a finitely generated $R$-submodule of $V$ for which $FL=V$. For an $R$-lattice $L$, the norm, scale and volume ideals of $L$, as defined in \cite[\S 82E]{OM}, will be denoted by $\mathfrak{n}(L)$, $\mathfrak{s}(L)$, and $\mathfrak{v}(L)$, respectively. The lattice $L$ will be said to be {\em integral} if $\mathfrak{s}(L) \subseteq R$ and {\em primitive} if $\mathfrak{s}(L) = R$. For a fractional $R$-ideal $\mathfrak{A}$ of $F$, an $R$-lattice $L$ of rank $n$ is {\em $\mathfrak{A}$-modular} if $\mathfrak{v}(L) = (\mathfrak{s}(L))^n$. An $R$-lattice is said to be {\em modular} if it is $\mathfrak{A}$-modular for some $\mathfrak{A}$, {\em unimodular} if it is $R$-modular, and {\em $a$-modular} for some $a\in R$ if it is $aR$-modular. For an $R$-lattice $L$ and $0\neq a\in F$, let $aL=\{av : v\in L\}$ and $L^{aR}=\{x\in L : B(x,L)\subseteq aR\}$. Note in particular that $\mathfrak{s}(aL)=a^2\mathfrak{s}(L)$ and $\mathfrak{n}(aL)=a^2\mathfrak{n}(L)$.

If $L$ is a free $R$-lattice with basis $\{x_1,\ldots,x_n\}$, then the {\em Gram matrix of $L$ with respect to $\{x_1,\ldots,x_n\}$} is the symmetric $n\times n$-matrix $(B(x_i,x_j))$. For a symmetric $n\times n$-matrix $M$, we will write $L\cong M$ to indicate that there exists a basis for $L$ such that $M$ is the Gram matrix of $L$ with respect to that basis. In particular, $L\cong \lan a_1,\ldots,a_n\ran$ will mean that $L$ has an orthogonal basis for which the Gram matrix is the diagonal matrix with the indicated diagonal entries. For a $\Z$-lattice $L$, all Gram matrices of $L$ have the same determinant; this common value is the {\em discriminant} of $L$, which will be denoted here by $d(L)$.  

For our purposes, the ring $R$ will always be either the ring $\Z$ of rational integers or the ring $\Z_p$ of $p$-adic integers for some $p\in S$, where $S$ denotes the set of rational primes. For a $\Z$-lattice $L$ and $p\in S$, $L_p$ will denote the $p$-adic localization of $L$; that is, $L_p = L\otimes _\Z\Z_p$.  For odd $p\in S$, a modular $\Z_p$-lattice can always be written as an orthogonal sum of rank $1$ sublattices; for $p=2$, such a lattice can be written as an orthogonal sum of modular sublattices of rank $1$ or $2$ \cite[93:15]{OM}. In the latter case, it is useful to recall from \cite[\S 93B]{OM} the notation $A(\alpha,\beta)$ to denote the unimodular matrix $\left( \begin{smallmatrix} \alpha & 1 \\ 1 & \beta \end{smallmatrix} \right)$, and $\lan A(\alpha,\beta)\ran$ to denote a unimodular lattice having $A(\alpha,\beta)$ as Gram matrix. We also introduce the special symbols $\mathbb A = \lan A(2,2) \ran$ and $\mathbb H = \lan A(0,0) \ran$. Note that if a $\Z_2$-lattice does not have an orthogonal basis, then it is split by a binary modular lattice $M$ such that $M\cong{ \lan\xi A(2,2)\ran}$ or $M\cong {\lan\xi A(0,0)\ran}$ for some nonzero $\xi{\in \Z_2}$ \cite[93:11]{OM}. For a $\Z$-lattice $L$, the Jordan splitting provides a decomposition of $L_p$ into an orthogonal sum of modular components of different scales. For a primitive $\Z$-lattice $L$, we will generally write a Jordan splitting of $L_p$ as 
\begin{equation}\label{localsplitting}
L_{(0)}\perp L_{(1)}\perp L_{(2)}\perp ...\perp L_{(t_p)},
\end{equation}
where each $L_{(i)}$ is $p^i$-modular or 0 with $L_{(t_p)}\neq 0$. The existence of such a splitting and the extent to which such splittings are unique are discussed in detail in \cite[\S 91C]{OM}.  

We will follow the definitions and terminology for the genus and spinor genus of a $\Z$-lattice as given in \cite[\S 102]{OM}. In particular, for a $\Z$-lattice $L$, cls($L$), spn($L$), $\spn^+(L)$ and gen($L$) will denote the class, spinor genus, proper spinor genus and genus of $L$, respectively. The numbers $h(L)$, $h_s(L)$ and $h^+_s(L)$ will denote the numbers of classes in the genus, spinor genus and proper spinor genus of $L$, respectively, and $g(L)$ and $g^+(L)$ the numbers of spinor genera and proper spinor genera in the genus of $L$, respectively. We will refer to $h(L)$ and $h_s(L)$ as the {\em class number} and {\em spinor class number} of $L$, respectively. Thus, our goal in this paper is to determine those $\Z$-lattices $L$ for which $h_s(L)=1$ and $h(L)>1$, or, equivalently, for which $h_s(L)=1$ and $g(L)>1$. Since all of the numbers $h(L)$, $h_s(L)$, $g(L)$ and $g^+(L)$ are invariant under scaling of the underlying bilinear mapping, we will generally restrict our attention to primitive lattices. 

The number $g^+(L)$ can be computed by means of an id\`{e}lic index formula given in \cite[102:7]{OM}. For the use of this formula, it is necessary to be able to explicitly compute the local spinor norm groups $\theta(O^+(L_p))$ for all $p\in S$, where $\theta$ denotes the spinor norm mapping and $O^+(L_p)$ denotes the group of rotations of $L_p$. These groups have been completely determined in the work of Kneser \cite{K61}, Hsia \cite{H75}, and Earnest and Hsia \cite{EH75}, and we will make frequent use of the results of those papers. Of particular use to us will be the fact that $g^+(L)=1$ holds whenever the containment 
\begin{equation}\label{spnnorm}
\Z_p^\times \subseteq \theta(O^+(L_p))
\end{equation}
holds for all $p\in S$, where $\Z_p^\times$ denotes the group of units of $\Z_p$. For odd $p\in S$ {(cf. \cite[102:9]{OM})}, the containment (\ref{spnnorm}) is known to hold whenever $L_p$ is split by a modular $\Z_p$-lattice of rank at least $2$ \cite[92:5]{OM}, and for $p=2$, it holds whenever $L_2$ is split by a modular $\Z_2$-lattice of rank at least $3$ \cite[Proposition A]{H75}. In particular, (\ref{spnnorm}) holds whenever $p$ does not divide $2d(L)$, since then $L_p$ is itself unimodular.

\section{Transformations that do not increase spinor class numbers}

An important ingredient in our arguments will be a family of transformations on the set of $\Z$-lattices on a given positive definite rational quadratic space defined by Gerstein \cite{G72}, following Watson \cite{W63}. These transformations have the property that they decrease the powers of primes occurring in the discriminant, while not increasing either the class number or spinor class number. In this section, we will review the definition and basic properties of these transformations, carry out some local computations that will be used throughout the remainder of the paper, and produce an initial list of possible prime divisors of the discriminants of one-class spinor genera. 

Let $L$ be a $\Z$-lattice on $V$ and let $p\in S$. Following \cite{G72}, we define a lattice $\mu_pL$ on $V$ by
\[
\mu_pL=L+p^{-1}L^{p^2\Z}.
\]
We will see that when $p^2 \mid d(L)$ the mapping taking $L$ to $\mu_pL$ often reduces the power of $p$ occurring in the discriminant and simplifies the $p$-adic structure of the lattice. Moreover, the transformation $\mu_p$ changes the lattice locally only at the prime $p$. For the case of a primitive $\Z$-lattice $L$, we record the effect on a Jordan splitting of $L_p$ of applying the $\mu_p$-transformation.

\begin{lem}\label{mu}
Let $p,q \in S$.  If $L_p$ has Jordan splitting (\ref{localsplitting}), then
\[
(\mu_p L)_q=\begin{cases}
L_q & \text{ for }q\neq p\\
L_{(0)}\perp L_{(1)}\perp p^{-1}\left(L_{(2)}\perp ...\perp L_{(t_p)}\right) &\text{ for }q=p.
\end{cases}
\]
\end{lem}

\begin{proof} See \cite[3.3]{G72}.
\end{proof}

For a positive integer $k$, the notation $\mu_p^k$ will denote the $k$-fold iterated application of the transformation $\mu_p$. From Lemma \ref{mu}, it follows that $\mu_p(L)\neq L$ if and only if $L_p$ has a $p^i$-modular component for some $i\geq 2$. Hence the sequence
 \begin{equation}\label{sequence}
 \{L,\mu_pL,\mu_p^2L,...\}
 \end{equation}
is eventually constant, terminating at a lattice for which the localization at $p$ contains at most a unimodular and $p$-modular component.  After localization, we obtain from (\ref{sequence}) a sequence 
\begin{equation}\label{local sequence}
\{L_p,(\mu_pL)_p,(\mu_p^2L)_p,...\}
\end{equation}
of $\Z_p$-lattices that is eventually constant. In \cite{G72}, the lattice $\mu L$ is defined to be the lattice on $V$ whose localization at every $p\in S$ is just the constant limit of sequence (\ref{local sequence}). For our purposes, we will define a related lattice $\muu L$ in such a way that, for any $p\in S$,
\begin{equation}\label{divisibility}
p\mid d(L) \text{ if and only if } p\mid d(\muu L).
\end{equation}
In order to do this, let $T$ denote the set of prime divisors of $d(L)$. Thus, 
\[
T= \{p\in S: L_p \text{ is not unimodular}\}.
\]
For $p\in T$, let $\hat L(p)$ denote the last non-unimodular $\Z_p$-lattice occurring in the sequence (\ref{local sequence}). Then define $\muu L$ to be the $\Z$-lattice on $V$ such that, for $p\in S$,
\[
(\muu L)_p=\begin{cases}
\hat L(p) &\text{ if }p\in T,\\
L_p & \text{ if }p\in S\setminus T.
\end{cases}
\]
With this definition, it can be seen that (\ref{divisibility}) holds, and the following important properties can be established as in \cite{G72}.

\begin{lem}
Let $L$ and $M$ be $\Z$-lattices on the same underlying rational quadratic space $V$.  Then the following hold. 
\begin{enumerate}
\item If $\gen(M)=\gen(L)$, then $\gen(\muu M)=\gen(\muu L)$. 
\item If $M\cong L$, then $\muu M\cong\muu L$. 
\item If $\spn(M)=\spn(L)$, then $\spn(\muu M)=\spn(\muu L)$. 
\end{enumerate}
\end{lem}

\begin{proof}
Proofs of (1) and (2) follow immediately as in the proof of \cite[Lemma 3.5]{G72}, and the proof of (3) follows similarly to that of (2), except replacing $\phi$ with $\phi\Sigma$ where $\phi\in O(V)$ and $\Sigma\in J'_V$, as defined in \cite[\S 101D]{OM}. 
\end{proof}

\begin{lem}\label{h_s}
Let $L$ be a $\Z$-lattice. Then
\begin{enumerate}
\item $h(\muu L)\leq h(L)$, and
\item $h_s(\muu L)\leq h_s(L)$.
\end{enumerate}
\end{lem}

\begin{proof}
Assertion (1) follows as in the proof of \cite[Theorem 3.6]{G72}, and (2) follows similarly, as noted in \cite{EH91}.
\end{proof}

Let $L$ be a $\Z$-lattice of rank $n$ and let $p \in S$. We define the {\em $p$-profile} of $L$ to be the non-decreasing $n$-tuple $(a_1,\ldots,a_n)_p$ of integers in which the integer $i$ appears $r_i$ times, where a Jordan splitting of $L_p$ contains a $p^i$-modular component of rank $r_i$. Since the ranks and scales of the Jordan components are invariants of the lattice, this notion is independent of the choice of the Jordan splitting. In particular, for a $\Z$-lattice for which  
\begin{equation}\label{Jsplitting}
L_p\cong \lan a,p^\beta b, p^\gamma c, p^\delta d\ran 
\end{equation}
where $a,b,c,d\in \Z_p^\times$ and $0\leq \beta\leq \gamma\leq \delta$, the $p$-profile of $L$ is $(0,\beta,\gamma,\delta)_p$. Moreover, $\muu L$ is defined so that the $p$-profile of $L$ is minimal, but not the all zeros tuple, at each prime $p\mid d(L)$. Therefore, since $\muu L$ is primitive, the possible resulting $p$-profiles are 
\begin{equation}\label{tuples}
\begin{array}{ccc}
(0,0,0,1)_p & (0,0,1,1)_p & (0,1,1,1)_p\\
(0,0,0,2)_p & (0,0,2,2)_p & (0,2,2,2)_p
\end{array}
\end{equation}
for any $p\mid d(L)$.  Starting from a $\Z$-lattice $L$ with localization at $p$ given by (\ref{Jsplitting}), repeated application of Lemma \ref{mu} then gives: 
\[
(\muu L)_p=\begin{cases}
\mu_p^\frac{\delta-1}{2}(L_p)\cong \lan a,p^{\beta'} b, p^{\gamma'} c, pd\ran    & \text{ if $\delta$ is odd,}\\
\mu_p^\frac{\delta}{2}(L_p)\cong\lan a,p^{\beta'} b, p^{\gamma'} c, d\ran   & \text{ if $\delta$ is even, and $\beta$ or $\gamma$ is odd,}\\
\mu_p^\frac{\delta-2}{2}(L_p)\cong\lan a,p^{\beta''} b, p^{\gamma''} c, p^2d\ran   & \text{ if $\beta$, $\gamma$ and $\delta$ are even.}
\end{cases}
\]
where $\beta',\gamma'\in \{0,1\}$ and $\beta'',\gamma''\in \{0,2\}$. 

The next result will be valuable for relating information about the discriminants of one-class spinor genera for quaternary lattices to those of one-class genera.

\begin{prop}\label{spngen}
If $L$ is a primitive quaternary $\Z$-lattice, then $\gen(\muu L)=\spn(\muu L)$. 
\end{prop}

\begin{proof}
For any odd prime $p$, $L_p$ has a Jordan splitting of the type (\ref{Jsplitting}) and so the $p$-profile of $\muu L$ is one of those listed in (\ref{tuples}). Hence $(\muu L)_p$ contains a modular component of rank at least 2, and consequently $\Z_p^\times \subseteq \theta(O^+((\muu L)_p))$.  When $p=2$, if $L_2$ is split by the scaling of $\mathbb A$ or $\mathbb H$, then $(\muu L)_2$ will also be split by some scaling of $\mathbb A$ or $\mathbb H$, and hence $\Z_2^\times\subseteq \theta(O^+((\muu L)_2))$ by \cite[Lemma 1]{H75}.   Otherwise, $L_2$ is diagonalizable, and hence $\muu L_2$ has 2-profile among (\ref{tuples}). The 2-profiles which admit a ternary modular component will give $\Z_2^\times\subseteq \theta(O^+((\muu L)_2))$ by \cite[93:20]{OM}, and the 2-profiles which admit binary modular components will give the same result by \cite[Theorem 3.14]{EH75}.
\end{proof}

\begin{cor}\label{corollary1}
If $h_s(L)=1$ then $h(\muu L)=1$. 
\end{cor}

\begin{proof}
If $h_s(L)=1$, it follows from Lemma \ref{h_s} that $h_s(\muu L)=1$.  Moreover, from Proposition \ref{spngen} we know that $\spn(\muu L)=\gen(\muu L)$; consequently $h(\muu L)=1$.  
\end{proof}

We are now in a position to begin the process of eliminating possible discriminants for one-class spinor genera by restricting the prime factors that can occur in such discriminants. For this purpose, let $\DD_s$ denote the set of discriminants of primitive positive definite quaternary quadratic $\Z$-lattices whose class and spinor genus coincide but whose class and genus do not coincide, and let $\PP_s$ denote the set of prime divisors of discriminants in $\DD_s$. Our goal will be to show that $\DD_s = \{729\}$, for then Theorem 1.1 will follow from \cite{N75} as noted previously. From an examination of the discriminants of the lattices appearing in \cite{KL}, it can be seen that 
\[
\PP=\{2,3,5,7,11,13,17,23\}
\]
is the set of all prime factors of the discriminants of primitive positive definite quaternary quadratic $\Z$-lattices with class number 1. It then follows from Corollary \ref{corollary1} and (\ref{divisibility}) that: 

\begin{cor}\label{corollary2}
If $p\in \PP_s$, then $p\in \PP$. 
\end{cor}

Based on this, we know that  
\[
\PP_s\subseteq \{2,3,5,7,11,13,17,23\}.
\]
It will also be useful to note that if $L$ is a primitive positive definite quaternary $\Z$-lattice $L$ with $h_s(L)=1$, then for any prime $p\not\in \PP$, $L_p$ is unimodular and hence (\ref{spnnorm}) holds at $p$.

\section{Bounding prime powers using the mass formula}

The remainder of the paper will be devoted to showing that $\DD_s = \{729\}$. Consequently, from this point on, the term {\em $\Z$-lattice} will always refer to a primitive positive definite quaternary $\Z$-lattice. Our next step will be to obtain bounds on the positive integers $n$ for which discriminants of the type $q^n$ could possibly be in $\DD_s$ for the primes $q=3,5,7$. For this purpose, we will use the Minkowski-Siegel mass formula to establish lower bounds on the mass $m(L)$ of a $\Z$-lattice $L$ for which $h_s(L)=1$ but $h(L)>1$. For the mass computations we follow the presentation of the mass formula given by Conway and Sloane in \cite{CS88}.  In general, the mass of a $\Z$-lattice $L$ is given by 
\[
m(L)={\sum_{i=1}^{h(L)}\frac{1}{\left|O(L^{(i)})\right|}}, 
\]
where $L^{(1)},\ldots,L^{(h(L))}$ are representatives of the distinct isometry classes in the genus of $L$ and $O(K)$ denotes the orthogonal group of the lattice $K$. Summing instead over representatives $L^{(1)},\ldots,L^{(h_s(L))}$ of the distinct isometry classes in the spinor genus of $L$, we obtain the spinor mass 
\[
m_s(L)=\sum_{i=1}^{h_s(L)}\frac{1}{\left|O(L^{(i)})\right|}. 
\]
From a proof analogous to that of \cite[Satz 1]{K61} it can be shown that the mass is distributed evenly over the spinor genera in the genus of any given $L$ (cf. \cite[p. 134]{EH91}, \cite[p. 14]{N}), so mass satisfies 
\[
m_s(L)=\frac{m(L)}{g(L)}.
\]
For any $p\in S$, consider the splitting of $L_p$ into its Jordan components, 
\[
L_p=L_{(-1)}\perp L_{(0)}\perp L_{(1)}\perp L_{(2)}\perp... 
\]
where each $L_{(i)}$ is $p^i$-modular or $0$.  When $p=2$ a component $L_{(i)}$ in this splitting is called {\em type I} if $\mathfrak{n}(L_{(i)})=\mathfrak{s}(L_{(i)})$ and {\em type II} if $\mathfrak{n}(L_{(i)})=2\mathfrak{s}(L_{(i)})$.  Using formula (2) from \cite{CS88}, specialized to the quaternary case, we have 
\[
m(L)=2\pi^{-5}\cdot \prod_{j=1}^4\Gamma\left(\frac{j}{2}\right)\cdot \prod_p2m_p(L)=\pi^{-4}\cdot \prod_p2m_p(L),
\]
where $m_p(L)$ is the local $p$-mass given by 
\[
m_p(L)=\prod_{-1\leq i}M_p\left(L_{(p^i)}\right)\cdot \prod_{-1\leq i<j}\left(p^{i-j}\right)^{\frac{1}{2}n(i)n(j)}\cdot 2^{n(I,I)-n(II)}.
\]
where $n(i)$ is the dimension of $L_{(i)}$.  We refer to the left-hand product in $m_p(L)$ as the {\em diagonal product} and the other product as the {\em cross product}.  The value $n(I,I)$ counts the number of adjacent pairs, $L_{(i)}$ and  $L_{({i+1})}$, that are both of type I, and $n(II)$ is the sum of all dimensions of type II components in the Jordan decomposition (the $n(I,I)$  and $n(II)$ values are only relevant in the case when $p=2$).  The components of $L_p$ are classified according to {\em species} given in Table 1 \cite{CS88}, which are determined by type and {\em octane value}, which is a measure of the square class of the discriminant of the component, defined formally on \cite[p. 265]{CS88} . 

\begin{lem}\label{massbound}
If $L$ is a $\Z$-lattice with $d(L)=q^n$ for some odd prime $q$ and positive integer $n$, and $L_q$ has 1-dimensional modular components $L_{(i)}$ for $i\in \{0,k,l,m\}$ with $0<k<l<m$, then
\[
m(L)\geq\frac{q^{\frac{3m+l-k}{2}}}{2^9\cdot 3\cdot 5}\cdot (1-q^{-2})^2
\]
when $n$ is even, and 
\[
m(L)>\frac{q^{\frac{3m+l-k}{2}}}{2^8\cdot 3^2\cdot 5}\cdot (1-q^{-4})
\]
when $n$ is odd.
\end{lem}

\begin{proof}
We will bound the mass of $L$ by first computing the local $p$-mass at each prime $p$.  When $p=2$, we have the $2$-adic splitting 
\[
L_2 = L_{({-1})}\perp L_{({0})}\perp L_{({1})}
\]
where both $L_{({-1})}$ and $L_{({1})}$ are 0-dimensional, and 
\begin{eqnarray}\label{L2}
L_{(0)}\cong \begin{cases}
\lan \epsilon_1,\epsilon_2,\epsilon_3,\epsilon_4\ran,\\
\mathbb A\perp \mathbb A\cong \mathbb H\perp \mathbb H, \text{ or}\\
\mathbb A\perp \mathbb H, 
\end{cases}
\end{eqnarray}
where $\epsilon_i\in \Z_2^\times$.  In the first case of (\ref{L2}), $L_{(0)}$ is free of type I.  Consequently both $L_{({-1})}$ and $L_{({1})}$ are 0-dimensional bound forms, and therefore each contributes $1/2$ to the diagonal product.  If $\epsilon_1\cdot \epsilon_2\cdot \epsilon_3\cdot \epsilon_4\equiv \pm 1\mod 8$, then $L_{(0)}$ has species $2+$, otherwise $L_{(0)}$ has species $2-$.  For the cross product, we have 
\[
\prod_{-1\leq i<j}\left(2^{i-j}\right)^{\frac{1}{2}n(i)n(j)}\cdot 2^{n(I,I)-n(II)}=1
\]
since there is only one component of non-zero dimension. Therefore, we obtain 
\begin{eqnarray}\label{L2typeI}
m_2(L)=\begin{cases}
\frac{1}{4} & \text{ if $L_{(0)}$ has species $2+$}\\
\frac{1}{12} & \text{ if $L_{(0)}$ has species $2-$}.
\end{cases}
\end{eqnarray}
In the second two cases of (\ref{L2}), $L_{(0)}$ is free of type II, and therefore $L_{({-1})}$ and $L_{({1})}$ are 0-dimensional free forms which only contribute $1$ to the diagonal product.  When $d(L_{(0)})\equiv 1\mod 8$ then $L_{(0)}$ has species $4+$, and when $d(L_{(0))}\equiv -3\mod 8$ then $L_{(0)}$ has species $4-$.  For the cross product, we get  
\[
\prod_{-1\leq i<j}\left(2^{i-j}\right)^{\frac{1}{2}n(i)n(j)}\cdot 2^{n(I,I)-n(II)}=2^{-4}
\]
since there is a single component which has non-zero dimension, and $n(II)=4$.  Therefore, we have 
\begin{eqnarray}\label{L2typeII}
m_2(L)=\begin{cases}
\frac{1}{18} & \text{ if $L_{(0)}$ has species $4+$}\\
\frac{1}{30} & \text{ if $L_{(0)}$ has species $4-$}.
\end{cases}
\end{eqnarray}
We have exhausted all possibilities for local structure at 2. Therefore combining equations (\ref{L2typeI}) and (\ref{L2typeII}), we can begin to bound $m(L)$ by  
\begin{eqnarray}\label{bound1}
m(L)\geq\pi^{-4}\cdot\frac{1}{3\cdot 5} \cdot \prod_{p\neq 2}2m_p(L).
\end{eqnarray}
When $p=q$, we are assuming that $L_{(i)}$ is 1-dimensional for $i\in \{0,k,l,m\}$ where $0< k<l<m$, and all other components are 0-dimensional.  The 1-dimensional terms each have species 1 and therefore each contributes $1/2$ to the diagonal product, and the 0-dimensional components all contribute 1.  For the cross product, we have 
\[
\prod_{-1\leq i<j}\left(q^{i-j}\right)^{\frac{1}{2}n(i)n(j)}=\left[\frac{q^k}{q^0}\cdot\frac{q^l}{q^0}\cdot\frac{q^l}{q^k}\cdot\frac{q^m}{q^0}\cdot\frac{q^m}{q^k}\cdot\frac{q^m}{q^l}\right]^{\frac{1}{2}}=q^\frac{3m+l-k}{2}.
\]
Therefore, combining the diagonal product and the cross product, we obtain 
\begin{eqnarray}\label{Lq}
m_q(L)=\frac{q^\frac{3m+l-k}{2}}{2^4}
\end{eqnarray}
and with this we can further improve upon (\ref{bound1}), obtaining
\begin{eqnarray}\label{bound2}
m(L)\geq\pi^{-4}\cdot\frac{q^\frac{3m+l-k}{2}}{2^3\cdot 3\cdot 5} \cdot \prod_{p\neq 2,q}2m_p(L).
\end{eqnarray}
If $p\neq 2,q$, then $L_p$ is unimodular and therefore $L_{(0)}$ is 4-dimensional, while all other components of $L_p$ are 0-dimensional.  Because of this, the cross product equals 1 for $m_p(L)$ whenever $p\neq 2,q$.  When $n$ is even, then $d(L)$ is always a quadratic residue modulo $p$, and therefore $L_{(0)}$ has genus $4+$, and 
\[
m_p(L)=\frac{1}{2(1-p^{-2})^2}.
\]
In this case we can refine the bound in (\ref{bound2}) to obtain 
\begin{eqnarray*}
m(L)&\geq&\pi^{-4}\cdot\frac{q^\frac{3m+l-k}{2}}{2^3\cdot 3\cdot 5} \cdot \prod_{p\neq 2,q}\frac{1}{(1-p^{-2})^2}
\end{eqnarray*}
and hence 
\[
m(L)\geq\pi^{-4}\cdot\frac{q^\frac{3m+l-k}{2}}{2^3\cdot 3\cdot 5} \cdot (1-2^{-2})^2\cdot (1-q^{-2})^2\cdot \zeta(2)^2=\frac{q^\frac{3m+l-k}{2}}{2^9\cdot 3\cdot 5}\cdot (1-q^{-2})^2
\]
which is the inequality we wanted to reach. 

On the other hand, when $n$ is odd, the local $p$-mass depends on the square class of $d(L)$.  If  $q$ is a quadratic residue modulo $p$ then $L_{(0)}$ has species $4+$, otherwise it has species $4-$, and thus
\begin{eqnarray*}\label{Lp}
m_p(L)=\begin{cases}
\frac{1}{2(1-p^{-2})^2} & \text{ when $L_{(0)}$ has species $4+$ }\\
\frac{1}{2(1-p^{-4})} & \text{ when $L_{(0)}$ has species $4-$.}
\end{cases}
\end{eqnarray*}
With the additional observation that 
\[
\frac{1}{(1-p^{-2})^2} > \frac{1}{(1-p^{-4})}, 
\]
we can further improve the bound (\ref{bound2}) on $m(L)$ by 
\begin{eqnarray*}\label{bound3}
m(L)&>&\pi^{-4}\cdot\frac{q^\frac{3m+l-k}{2}}{2^3\cdot 3\cdot 5} \cdot \prod_{p\neq 2,q}\frac{1}{(1-p^{-4})}
\end{eqnarray*}
and thus 
\[
m(L)>\pi^{-4}\cdot\frac{q^\frac{3m+l-k}{2}}{2^3\cdot 3\cdot 5} \cdot (1-2^{-4})\cdot (1-q^{-4})\cdot \zeta(4)=\frac{q^\frac{3m+l-k}{2}}{2^8\cdot 3^2\cdot 5}\cdot (1-q^{-4})
\]
which is the desired inequality. 
\end{proof}

Using this bound on $m(L)$ we can begin to bound the powers of certain primes appearing in the discriminant of a lattice $L$ having a one-class spinor genus.  

\begin{lem}\label{qbounds}
Let $L$ be a $\Z$-lattice with $d(L)=q^n$ for some odd prime $q$ and positive integer $n$.  If $h_s(L)=1$ and $h(L)>1$ then $L_q$ has 1-dimensional modular components $L_{(i)}$ for $i\in \{0,k,l,m\}$ with $0<k<l<m$ and
\[
3m+l-k\leq\begin{cases}
 16 & \text{ for $q=3$ and $n$ even,}\\
 17 & \text{ for $q=3$ and $n$ odd,}\\
 11 & \text{ for $q=5$,}\\
  9 & \text{ for $q=7$.}
\end{cases}
\]
\end{lem}

\begin{proof}
Since $L_p$ is unimodular at every prime $p\neq q$, it follows that $L_q$ does not contain any modular components of rank larger than 1, or else (\ref{spnnorm}) holds at every $p\in $, which would mean $\gen(L)=\spn(L)$ and hence $h(L)=1$, contrary to assumption.  Since $h(L)>1$, we may conclude that $\gen(L)$ splits into multiple spinor genera, and since $q$ is odd, we may say more precisely that $g(L)=2$. Since $\left|O(L)\right|\geq 2$, if we can show that $m(L)>1$, then we will have shown that $m_s(L)>1/2$, and consequently the sum 
\[
m_s(L)=\sum_{i=1}^{h_s(L)}\frac{1}{\left|O(L^{(i)})\right|}
\]
must be taken over more than one class.  In other words, if $m(L)>1$ then $h_s(L)>1$. 

From Lemma \ref{massbound}, in order to show that $m(L)>1$, it suffices to have 
\begin{eqnarray}\label{evenbound}
\frac{q^{\frac{3m+l-k}{2}}}{2^9\cdot 3\cdot 5}\cdot (1-q^{-2})^2\geq1
\end{eqnarray}
for $n$ even, and 
\begin{eqnarray}\label{oddbound}
\frac{q^{\frac{3m+l-k}{2}}}{2^8\cdot 3^2\cdot 5}\cdot (1-q^{-4})>1
\end{eqnarray}
for $n$ odd. This leads to the bounds in the statement, completing the proof. 
\end{proof}

\section{Reducing the list of possible prime divisors}

The goal of this section is to eliminate all the potential primes from $\PP_s$ except for 2 and 3. This will be done in a series of lemmas. 

\begin{lem}\label{lem:17}
$\PP_s\subseteq\{2,3,5,7\}$
\end{lem}

\begin{proof}
To prove this claim, we need to show that $11,13,17$ and $23$ all fail to appear in $\PP_s$.  The general strategy will be as follows.  For a given $p\in \{11,13,17, 23\}$ we will suppose that $L$ is a $\Z$-lattice for which $h_s(L)=1$, $h(L)>1$ and $p\mid d(L)$.  We know from Proposition \ref{spngen} that $h(\muu L)=1$ and therefore $\muu L$ appears in the table of 481 lattices in \cite{KL}. Moreover, from (\ref{divisibility}), we know that $p\mid d(\muu L)$, so we can narrow down the possible candidates for $\muu L$.  For each candidate, we will consider the associated $p$-profiles $(k,l,m,n)_p$ for $\muu L$.  From here, we will define $L'$ to be a lattice which descends to $\muu L$ by one iteration of the $\mu_p$-transformation, that is, $\mu_p(L')=\muu L$.  Next, we will examine the $p$-profiles of such $L'$, taking note that $h_s(L')=1$ and $L'_q\cong \muu L_q$ for every $q\neq p$.  In most cases, we will see that the $p$-profile of $L'$ forces $h(L')=1$ and hence $\gen(L')=\spn(L')$, meaning $L'$ is among the lattices in \cite{KL}, which will lead to a contradiction. 

%
%

Suppose first that $p=11$. Then $\muu L$ corresponds to one of 9 lattices in \cite{KL} having a discriminant divisible by 11.  These lattices all have discriminants of the form $2^r\cdot 3^s\cdot 11^t$ for non-negative integers $r,s,t$, and $11$-profiles from among 
\[
\begin{array}{ccc}
(0,0,0,1)_{11} & (0,0,1,1)_{11} & (0,1,1,1)_{11}.
\end{array}
\]
Letting $L'$ be a lattice for which $\mu_{11}(L')=\muu L$, then when $\muu L$ has 11-profile $(0,0,0,1)_{11}$, $L'$ has 11-profile from among 
\[
\begin{array}{ccccc}
(0,0,1,2)_{11} & (0,0,0,3)_{11} &(0,1,2,2)_{11} & (0,0,2,3)_{11} & (0,2,2,3)_{11},
\end{array}
\]
and when $\muu L$ has 11-profile $(0,1,1,1)_{11}$ then $L'$ has 11-profile from among 
\[
\begin{array}{ccc}
(0,1,1,3)_{11} & (0,1,3,3)_{11} & (0,3,3,3)_{11}.
\end{array}
\]
In all of these cases, $L_{11}'$ contains a modular component of rank at least 2, so we know that $\Z_{11}^\times\subseteq \theta(O^+(L'_{11}))$, and since $L'_p\cong \muu L_p'$ for every prime $p\neq 11$, we have $\Z_{p}^\times\subseteq \theta(O^+(L'_p))$ for every $p\in S$.  Consequently $\gen(L')=\spn(L')$ and $h_s(L')=1$, implying that $h(L')=1$.  However, this is impossible since it can be checked that the list of lattices in \cite{KL} doesn't contain any lattice admitting such an 11-profile.  On the other hand, when $\muu L$ has 11-profile $(0,0,1,1)_{11}$, there is only one lattice in \cite{KL} with such an 11-profile, and it has the local structure 
\[
(\muu L)_{11}\cong \lan 1,\btu_{11}\ran\perp 11\lan 1,\btu_{11}\ran  
\]
where for an odd $p\in S$, $\btu_p$ denotes a non-square unit in $\Z_p^\times$.  This implies that $L_{11}'$ will either contain a binary modular component, or will be split by a sublattice of the form 
\[
\lan 1\ran \perp {11^2}\lan \btu_{11}\ran \qquad\text{or}\qquad \lan \btu_{11}\ran \perp  {11^2}\lan 1\ran.
\] 
and so in any case $\Z_{11}^\times\subseteq \theta(O^+(L'_{11}))$.  Therefore, again we have $h(L')=1$ and $h_s(L')=1$, but $L'$ must have 11-profile from among, 
\[
\begin{array}{ccccc}
(0,0,3,3)_{11} & (0,0,1,3)_{11} & (0,1,2,3)_{11} & (0,1,1,2)_{11} & (0,2,3,3)_{11},
\end{array}
\]
and no such 11-profile appears in \cite{KL}.  Thus, we may conclude that $11\not\in \PP_s$.  

%
%

When $p=13$ the argument proceeds similarly.  In this case $\mu L$ must correspond to one of only 3 lattices among those in \cite{KL} which have discriminants divisible by 13.  These have discriminants $13,13^2$ and $13^3$, and $13$-profiles 
\[
\begin{array}{lll}
(0,0,0,1)_{13} & (0,0,1,1)_{13} & (0,1,1,1)_{13},
\end{array}
\]
respectively.  Moreover, when $\muu L$ has 13-profile $(0,0,1,1)_{13}$, we have 
\[
(\muu L)_{13}\cong \lan 1,\btu_{13}\ran \perp 13\lan 1,\btu_{13}\ran. 
\] 
From here the argument proceeds precisely as above, and we conclude that $13\not\in \PP_s$.

%
%

When $p=17$ the argument is further simplified, since here $d(\muu L)$ is either 17 or $17^3$ with respective 17-profiles, 
\[
\begin{array}{lll}
(0,0,0,1)_{17} & (0,1,1,1)_{17},
\end{array}
\]
from which we can deduce that $17\not\in \PP_s$. 

%
%

When $p=23$, $d(\muu L)$ is one of $3\cdot 23, 3^3\cdot 23, 3\cdot 23^3$ or $3^3\cdot 23^3$ with respective $23$-profiles 
\[
\begin{array}{cc}
(0,0,0,1)_{23} & (0,1,1,1)_{23}.
\end{array}
\]
From here we can use precisely the argument as used above to conclude that $23\not\in \PP_s$. 
\end{proof}

\begin{lem}
$7\not\in \PP_s$
\end{lem}

\begin{proof}
Suppose that $7\in \PP_s$, and suppose that $L$ is a $\Z$-lattice which has $7\mid d(L)$, $h_s(L)=1$ and $h(L)>1$.  Then $\muu L$ must be among the lattices appearing in \cite{KL}, which can have either odd or even discriminants.  We will consider these cases separately.  The proof will proceed similarly to the proof of Lemma \ref{lem:17}, where we define $L'$ to be the lattice which descends to $\muu L$ by one iteration of the $\mu_7$-transformation, and thus $\mu_7(L')=\muu L$ and $L'_p\cong \muu L_p$ for every $p\neq 7$. 

%
%

Consider first the case that $d(\muu L)$ is odd.  Then $d(\muu L)=3^k\cdot 7^m$ for non-negative integers $k$ and $m$, and possible $7$-profiles 
\[
\begin{array}{ccc}
(0,0,0,1)_7 & (0,0,1,1)_7 & (0,0,0,1)_7,
\end{array}
\]
where, in particular, $\muu L$ corresponding to the profile $(0,0,1,1)_7$ has discriminant $7^2$.  When $\muu L$ has 7-profile $(0,0,0,1)_7 $, $L'$ has 7-profile from among
\[
\begin{array}{ccccc}
(0,0,1,2)_{7} & (0,0,0,3)_{7} &(0,1,2,2)_{7} & (0,0,2,3)_{7} & (0,2,2,3)_{7},
\end{array}
\]
and when $\muu L$ has 7-profile $(0,1,1,1)_{7}$, $L'$ has 7-profile from among 
\[
\begin{array}{ccc}
(0,1,1,3)_{7} & (0,1,3,3)_{7} & (0,3,3,3)_{7}.
\end{array}
\]
But in all of these cases $L'_{7}$ contains a modular component of rank at least 2, which means that $h(L')=1$ and $h_s(L')=1$, leading to a contradiction, since no lattices among the list in \cite{KL} admit such 7-profiles.  On the other hand, when $\muu L$ has 7-profile $(0,0,1,1)_7$, $L'$ has 7-profile from among 
\[
\begin{array}{ccccc}
(0,0,3,3)_{7} & (0,0,1,3)_{7} & (0,2,3,3)_{7} & (0,1,2,3)_{7} & (0,1,1,2)_{7}.
\end{array}
\]
The first three of these can be immediately ruled out since they imply $h(L')=1$, but no lattices among those in \cite{KL} admit such 7-profiles.  When $L'$ has 7-profile $(0,1,2,3)_7$ we know that $d(\mu_7(L'))=7^2$ and hence $d(L')=7^6$; therefore we may conclude from Lemma \ref{qbounds} that $h_s(L')>1$, since $3(3)+2-1>9$.  But $h_s(L')\leq h_s(L)=1$, so this leads to a contradiction.  If $L'$ has 7-profile $(0,1,1,2)_7$, then $h(L')=1$, and $L'$ corresponds to the unique lattice in \cite{KL} with such a 7-profile, which has $d(L')=7^4$.  In this case we define $L''$ to be the lattice which descends to $L'$ by one iteration of the $\mu_7$-transformation.  Then $d(L'')$ must be a power of 7, and $L''$ must have a 7-profile from among 
\[
\begin{array}{ccc}
(0,1,1,4)_7 & (0,3,3,4)_7 & (0,1,3,4)_7.
\end{array}
\]
The first two cases can be immediately ruled out since they contain a binary modular component but do not appear in \cite{KL}, and the third case can be ruled out by Lemma \ref{qbounds} since $3(4)+3-1>9$.  

%
%

Now consider the case that $d(\muu L)$ is even.  Then $d(\muu L)=2^k\cdot 7^m$ for non-negative integers $k$ and $m$, and possible $7$-profiles 
\[
\begin{array}{ccc}
(0,0,0,1)_7 & (0,0,1,1)_7 & (0,1,1,1)_7.
\end{array}
\]
By the same argument used in the odd case, we can immediately rule out $7$-profiles $(0,0,0,1)_7$ and $(0,1,1,1)_7$.  Suppose that $\muu L$ has 7-profile $(0,0,1,1)_7$.  There is a unique lattice among the lattices in \cite{KL} with 7-profile $(0,0,1,1)_7$ and even discriminant; this lattice has 2-adic structure 
\[
(\muu L)_2\cong \mathbb A\perp 2^2\mathbb A.
\] 
Now $L'$ must have 7-profile from among 
\[
\begin{array}{ccccc}
(0,1,1,2)_7 & (0,0,1,3)_7 & (0,0,3,3)_7& (0,2,3,3)_7&(0,1,2,3)_7
\end{array}
\]
the first four of which can be immediately ruled out since \cite{KL} does not contain any lattices with even discriminant divisible by 7 admitting such a 7-profile.  On the other hand, when $L'$ has 7-profile $(0,1,2,3)_7$, we will use the mass formula to show that $h_s(L')>1$.  Since $L'_2\cong (\muu L)_2$ we have 
\[
m_2(L')=m_2(\muu L)=\frac{1}{3}\cdot \frac{1}{3}\cdot \left(\frac{2^2}{2^0}\right)^{\frac{1}{2}\cdot 2\cdot 2}\cdot 2^{-4}=\frac{1}{3^2}
\]
and since $d(L')=2^4\cdot 7^6$, we have 
\begin{eqnarray*}
m(L')=\pi^{-4}\cdot 2\cdot m_2(L')\cdot 2\cdot m_7(L')\cdot (1-2^{-2})^2\cdot (1-7^{-2})^2\cdot \zeta(2)^2=7>1.
\end{eqnarray*}
Since $m(L')>1$ it follows that $m_s(L')>1/2$ and hence $h_s(L')>1$, leading to a contradiction.  
\end{proof}

\begin{lem}
$5\not\in \PP_s$. 
\end{lem}

\begin{proof}
Suppose that $L$ is a $\Z$-lattice with $5\mid d(L)$, $h_s(L)=1$ and $h(L)>1$.  Then $5\mid d(\muu L)$ and $\muu L$ must be from among the lattices in \cite{KL}.  But among the lattices in \cite{KL}, all discriminants divisible by 5 are of the form $2^r\cdot 3^s\cdot 5^t$ where $r,s$ are non-negative integers, and $t>0$.  Moreover, we know that the 5-profile of $\muu L$ must be from among 
\[
\begin{array}{ccc}
(0,0,0,1)_5 & (0,0,1,1)_5 & (0,1,1,1)_5.
\end{array}
\]
When $\muu L$ has 5-profile $(0,0,0,1)_5 $, $L'$ has 5-profile from among
\[
\begin{array}{ccccc}
(0,0,1,2)_{5} & (0,0,0,3)_{5} &(0,1,2,2)_{5} & (0,0,2,3)_{5} & (0,2,2,3)_{5},
\end{array}
\]
and when $\muu L$ has 5-profile $(0,1,1,1)_{5}$, $L'$ has 5-profile from among 
\[
\begin{array}{ccc}
(0,1,1,3)_{5} & (0,1,3,3)_{5} & (0,3,3,3)_{5}.
\end{array}
\]
We can immediately rule out all possible 5-profiles for $L'$ except for $(0,0,1,2)_5$ and $(0,1,2,2)_5$ since they correspond to $L'$ with $h(L')=1$, but no lattices in \cite{KL} admit such 5-profiles. If $L'$ has 5-profile $(0,0,1,2)_5$, then $h(L')=1$ and hence $L'$ must be one of two lattices in \cite{KL} with this profile, both of which have discriminant $d(L')=5^3$.  If we define $L''$ to be the lattice that descends to $L'$ by one iteration of the $\mu_5$-transformation, then $\mu_5(L'')=L'$ and $L''_p\cong L'_p\cong (\muu L)_p$ for every prime $p\neq 5$.   Consequently, $d(L'')$ is a power of $5$, and $L''$ has a 5-profile from among 
\[
\begin{array}{cccc}
(0,0,1,4)_5 & (0,0,3,4)_5 & (0,1,2,4)_5 & (0,2,3,4)_5
\end{array}
\]
and hence by Lemma \ref{qbounds} in all of these cases $h_s(L'')>1$.  Similarly, when $L'$ has 5-profile $(0,1,2,2)_5$ then we know that $h(L')=1$ and again $L'$ must be one of two lattices in \cite{KL} with this 5-profile, both of which have $d(L')=5^5$.  Hence, $d(L'')$ is a power of 5, and $L''$ has a 5-profile from among 
\[
\begin{array}{cc}
(0,1,4,4)_5 & (0,3,4,4)_5
\end{array}
\]
so again it follows from Lemma \ref{qbounds} and $h_s(L'')>1$. 

Suppose that $\muu L$ has 5-profile $(0,0,1,1)_5$.  Then $d(\muu L)=5^2$ or $2^2\cdot 5^2$, and $\muu L$ is one of 4 possible lattices in \cite{KL} which have $m_2(\muu L)$ equal to $1/4, 1/8$ or $1/36$ (computed using Sagemath \cite{Sage}). If we define $L'$ as in the previous paragraph, then $L'$ has a 5-profile from among 
\[
\begin{array}{ccccc}
(0,0,3,3)_{5} & (0,0,1,3)_{5} & (0,2,3,3)_{5} & (0,1,2,3)_{5} & (0,1,1,2)_{5}.
\end{array}
\]
and in the usual way the first three of these 5-profiles can be eliminated.  When $L'$ has 5-profile $(0,1,2,3)_5$, then
\[
m_2(L')=m_2(\muu L)
\]
and 
\[
m_5(L')=\frac{5^\frac{3(3)+2-1}{2}}{2^4}=\frac{5^5}{2^4}
\]
and hence  
\[
m(L')=\pi^{-4}\cdot 2\cdot m_2(L')\cdot 2\cdot m_5(L')\cdot (1-2^{-2})^2\cdot (1-5^{-2})^2\cdot \zeta(2)^2=m_2(\muu L)\cdot \frac{3^2\cdot 5}{2^2}.
\]  
But for any possible choice of $m_2(\muu L)$, it follows that $m_s(L')=m(L')/g(L')$ is not of the form $1/\mid O(L')\mid$, and hence $h_s(L')>1$.  When $L'$ has 5-profile $(0,1,1,2)_5$, then $h(L')=1$, so we define $L''$ to be the lattice which descends to $L'$ by one iteration of the $\mu_5$-transformation.  Then $L''$ has 5-profile from among 
\[
\begin{array}{ccc}
(0,1,1,4)_5 & (0,3,3,4)_5 & (0,1,3,4)_5
\end{array}
\]
and again we can immediately rule out the first two 5-profiles by the usual method.  When $L''$ has 5-profile $(0,1,3,4)_5$, then $m_2(L'')=m_2(L')=m_2(\muu L)\geq 1/36$, and 
\[
m_5(L'')=\frac{5^\frac{3(4)+3-1}{2}}{2^4}=\frac{5^7}{2^4}.
\]
Hence
\[
m(L'')=\pi^{-4}\cdot 2\cdot m_2(L'')\cdot 2\cdot m_5(L'')\cdot (1-2^{-2})^2\cdot (1-5^{-2})^2\cdot \zeta(2)^2\geq \frac{5^3}{2^4}.
\]
Since $m(L'')>1$ we may conclude that $m_s(L'')>1/2$ and hence $h_s(L'')>1$.
\end{proof}

\section{Completion of proof}

The list of possible primes in $\PP_s$ has now been reduced to 2 and 3. In this section, the proof of Theorem 1.1 will be completed in two lemmas, the first dealing with possible discriminants of the type $2^n$ and the second dealing with remaining discriminants of the type $2^k\cdot 3^{\ell}$. 

In the proof of the following lemma we will make use of results on the computation of the $2$-adic spinor norm groups from \cite{EH75}. We remind the reader that in the terminology of that paper, a $\Z_2$-lattice $M$ is said to have {\em even order} if $Q(P(M))\subseteq \Z_2^\times{\Q_2^\times}^2$ and {\em odd order} if  $Q(P(M))\subseteq 2\Z_2^\times{\Q_2^\times}^2$, where $P(M)$ denotes the set of all primitive anisotropic vectors whose associated symmetries are in $O(M)$.  A unary modular component, $2^m\lan \epsilon\ran$ where $\epsilon\in \Z_2^\times$, has odd or even order according to the parity of $m$.  Recall also that from \cite[Proposition 3.2]{EH75}, a binary unimodular $\Z_2$-lattice $M$ has even order if and only if $M\cong {\lan A(1,0)\ran}$ or ${\lan A(1,4\epsilon)\ran}$ and $M$ has odd order if and only if $M\cong {\lan A(0,0)\ran}$ or ${\lan A(2,2\epsilon)\ran}$ where $\epsilon\in \Z_2^\times$. It follows that any binary unimodular $\Z_2$-lattices which is neither odd nor even is isometric to one of the following:
\begin{eqnarray}\label{notoddoreven}
\lan 1,1\ran, \lan 3,3\ran, \lan 3,7\ran \text{ or } \lan 1,5\ran.  
\end{eqnarray}
We refer the reader to \cite[p. 531]{EH78} for the definition of {\em type E}, particularly noting that when a $\Z_2$-lattice is of type E, its spinor norm group contains the full group of units. 


\begin{lem}\label{2powers}
There are no powers of 2 appearing in $\mathcal D_s$. 
\end{lem}

\begin{proof}
Suppose that $L$ is a $\Z$-lattice with $d(L)=2^n$ for some $n>0$, and suppose that $h_s(L)=1$ while $h(L)>1$. Then $L_p$ is unimodular and thus (\ref{spnnorm}) holds for every odd $p\in S$. Hence, we may conclude that $L_2$ does not contain any improper modular components or modular components of dimension 3 or 4, since otherwise $\gen(L)=\spn(L)$ and thus $h(L)=1$.  Therefore, $L$ must have a 2-adic splitting
\[
L_2\cong \lan \epsilon_1\ran\perp 2^k\lan \epsilon_2\ran \perp 2^l\lan \epsilon_3\ran\perp 2^m\lan \epsilon_4\ran
\] 
where $\epsilon_i\in \Z_2^\times$, and $0\leq k\leq l\leq m$. In the remainder of the proof, we consider the various possible cases for $k,l$ and $m$.

Case I: $k=0$ and $l=m$. So $L_2$ has two binary modular components, 
\[
N\cong \lan \epsilon_1, \epsilon_2\ran\text{ and }M\cong 2^m\lan \epsilon_3, \epsilon_4\ran,
\]
where $m>0$.  According to \cite[Theorem 3.14]{EH75}, if $M$ and $N$ both have odd order, both have even order, or one of each, then (\ref{spnnorm}) holds for $p=2$ by \cite[Theorem 3.14 (i) and (ii)]{EH75}, and hence $\gen(L)=\spn(L)$, implying $h(L)=1$.  Therefore we may suppose that one of $M$ or $N$ must be from among the lattices in (\ref{notoddoreven}). If the other lattice has odd or even order, then (\ref{spnnorm}) holds for $p=2$ by \cite[Theorem 3.14 (iii)]{EH75}.  Therefore we may suppose that both $M$ and $N$ have neither odd nor even order.  Therefore, $M\cong \lan\epsilon_1,\epsilon_2\ran$ and $N\cong \lan \epsilon_3,\epsilon_4\ran$ where $\epsilon_i\in \Z_2^\times$.  If $M\not\cong N$ or if $m<4$, then (\ref{spnnorm}) holds for $p=2$ by \cite[Theorem 3.14 (iv)]{EH75}.  Thus, in order to simultaneously have $h_s(L)=1$ and $h(L)>1$, we may assume that $M\cong N$ is from among the binary forms in (\ref{notoddoreven}) and $m\geq 4$.  In that case, the mass of the genus is given by 
\[
m(L)=\pi^{-4}\cdot 2\cdot m_2(L)\cdot \prod_{p\neq 2}2m_p(L).
\]
Since 
\[
L_2\cong L_{(-1)}\perp L_{(0)}\perp _{(1)}\perp...\perp L_{(m-1)}\perp L_{(m)}\perp L_{(m+1)}, 
\]
each $L_{(i)}$ contributes $1/2$ to the diagonal product for $i\in \{-1,1,m-1,m+1\}$.  Moreover, since $L_{(0)}$ and $L_{(m)}$ are both 2-dimensional free type I forms with octane value $\pm2$, they have species 1 and therefore each contribute $1/2$ to the diagonal product.  Therefore, 
\[
m_2(L)=\left(\frac{1}{2}\right)^6\cdot \left(\frac{2^m}{2^0}\right)^{\frac{1}{2}\cdot 2\cdot 2}\cdot 2^{0-0}=2^{2m-6}
\]
and since $d(L_2)=2^{2m}\in {\Q^\times}^2$ we have 
\begin{eqnarray*}
m(L)=\pi^{-4}\cdot 2^{2m-5}\cdot (1-2^{-2})^2\cdot \zeta(2)^2=2^{2m-11}
\end{eqnarray*}
Since $g(L)=2 \text{ or } 4$, it follows that $m_s(L)=m(L)/g(L)>1/2$ for any $m\geq 7$, and thus $m_s(L)$ is not of the form $1/\left|O(L)\right|$, implying that $h_s(L)>1$.  

On the other hand, when $m=4,5 \text{ or }6$ the algorithm from \S \ref{algorithm}) can be used to determine all possible genera for lattices with 2-profile $(0,0,m,m)_2$.  When $m=4$, the algorithm produces 4 genera with 2-profile $(0,0,4,4)_2$ and 2-adic structure $M\perp 2^4 M$, with representative lattices
\[
L_1\cong \left[\begin{array}{cccc}
2 & 0 & 1 & -2 \\
0 & 2 & 1 & -2 \\
1 & 1 & 5 & -2 \\
-2 & -2 & -2 & 20
\end{array}\right] 
L_2\cong \left[\begin{array}{cccc}
1 & 0 & 0 & 0 \\
0 & 4 & 2 & 4 \\
0 & 2 & 5 & 2 \\
0 & 4 & 2 & 20
\end{array}\right]
\]
\[
L_3\cong \left[\begin{array}{cccc}
8 & 0 & 2 & 4 \\
0 & 2 & -1 & 0 \\
2 & -1 & 3 & 1 \\
4 & 0 & 1 & 10
\end{array}\right]
L_4\cong \left[\begin{array}{cccc}
3 & 0 & 0 & -1 \\
0 & 12 & -2 & 6 \\
0 & -2 & 3 & -1 \\
-1 & 6 & -1 & 6
\end{array}\right].
\]
Checking the structure of each $\gen(L_i)$ in Magma we see that only $\gen(L_1)$ splits into multiple spinor genera, both containing multiple classes, and for the remaining cases $\gen(L_i)=\spn(L_i)$.  When $m=5$, there are only 3 possible genera with profile $(0,0,5,5)_2$ and 2-adic structure $M\perp 2^5 M$, and they have representative lattices 
\[
L_1\cong \left[\begin{array}{cccc}
3 & 1 & -1 & -1 \\
1 & 4 & -2 & -2 \\
-1 & -2 & 4 & 4 \\
-1 & -2 & 4 & 36
\end{array}\right]
L_2\cong \left[\begin{array}{cccc}
2 & 0 & -1 & -2 \\
0 & 2 & 1 & -2 \\
-1 & 1 & 9 & 0 \\
-2 & -2 & 0 & 36
\end{array}\right]
\]
\[
L_3\cong \left[\begin{array}{cccc}
3 & -2 & 0 & 0 \\
-2 & 12 & 0 & 0 \\
0 & 0 & 3 & 1 \\
0 & 0 & 1 & 11
\end{array}\right].
\]
Of these genera, $\gen(L_1)=\spn(L_1)$, while $\gen(L_2)$ and $\gen(L_3)$ split into multiple spinor genera, each containing several classes.  When $m=6$, there are 4 possible genera with 2-profile $(0,0,6,6)_2$ and 2-adic structure $M\perp 2^6M$, with representative lattices 
\[    	
L_1\cong \left[\begin{array}{cccc}
4 & -2 & 0 & 0 \\
-2 & 5 & -2 & 0 \\
0 & -2 & 5 & 0 \\
0 & 0 & 0 & 64
\end{array}\right]
L_2\cong \left[\begin{array}{cccc}
2 & 0 & 1 & -2 \\
0 & 8 & 2 & 4 \\
1 & 2 & 9 & 0 \\
-2 & 4 & 0 & 36
\end{array}\right]
\]
\[
L_3\cong \left[\begin{array}{cccc}
2 & -1 & 0 & 0 \\
-1 & 6 & 1 & 0 \\
0 & 1 & 6 & 0 \\
0 & 0 & 0 & 64
\end{array}\right]
L_4\cong \left[\begin{array}{cccc}
6 & 0 & 1 & 6 \\
0 & 6 & 3 & 2 \\
1 & 3 & 7 & 2 \\
6 & 2 & 2 & 28
\end{array}\right]
\]
Again, a check of these genera in Magma reveals that $\gen(L_1)=\spn(L_1)$ and $\gen(L_3)=\spn(L_3)$, while the remaining genera split into two spinor genera, each containing several classes.  

Case II: $k=0$ and $0<l<m$. So 
\[
L_2\cong \lan \epsilon_1,\epsilon_2\ran\perp 2^l\lan \epsilon_3\ran\perp 2^m\lan \epsilon_4\ran.   
\]
The unary components are either odd or even according to the parity of $l$ and $m$. If the binary component is either odd or even, then in any case (\ref{spnnorm}) holds for $p=2$ by \cite[Theorem 3.14 (i) and (ii)]{EH75}.  Therefore we may suppose that the binary component is neither odd nor even, and hence is one of the lattices in (\ref{notoddoreven}).   If $l<4$ or if $m-l<4$ then (\ref{spnnorm}) holds for $p=2$ by \cite[Theorem 3.14 (iv)]{EH75}. Therefore we may assume that $l\geq 4$ and $k\geq 8$. Now we will compute the mass $m(L)$.  To compute the diagonal product, we have a decomposition 
\begin{eqnarray*}
L_2\cong L_{(-1)}\perp L_{(0)}\perp _{(1)}\perp...\perp L_{(l-1)}\perp L_{(l)}\perp L_{(l+1)}\perp ...&\\
...\perp L_{(m-1)}\perp L_{(m)}\perp L_{(m+1)},& 
\end{eqnarray*}
where each of the 0-dimensional forms is bound since it is adjacent to a form of type I, and therefore each $L_{(i)}$ contributes $1/2$ to the diagonal product for $i\in \{-1,1,l-1,l+1,m-1,m+1\}$.  Moreover, the binary part is free of type I with octane value $\pm 2$, and therefore has species 1, and the two unary parts are free of type 1 with octane value $\pm 1$ and therefore have species $0+$.  Therefore, computing the local mass we have 
\[
m_2(L)=\left(\frac{1}{2}\right)^7\cdot (2^{m-l})^{\frac{1}{2}\cdot 1\cdot 1}\cdot (2^{m})^{\frac{1}{2}\cdot 2\cdot 1}\cdot (2^{l})^{\frac{1}{2}\cdot 2\cdot 1}=2^{\frac{3m+2l-14}{2}}\geq 2^9,
\]
and hence 
\begin{eqnarray*}
m(L)& > & \pi^{-4}\cdot 2^{10}\cdot (1-2^{-4})\cdot \zeta(4)=\frac{2^5}{3}.
\end{eqnarray*}
But since $g(L)=2 \text{ or }4$, this implies that $m_s(L)=h(L)/g(L)>1/2$, and therefore we may conclude that $h_s(L)>1$.  

Case III: $0<k$ and either $k=l$ or $l=m$. So
\[
L_2\cong \begin{cases}
\lan \epsilon_1\ran \perp 2^k\lan \epsilon_2, \epsilon_3\ran\perp 2^m\lan \epsilon_4\ran  & \text{ when }k=l\\
\lan \epsilon_1\ran \perp 2^k\lan \epsilon_2\ran\perp 2^m\lan\epsilon_3, \epsilon_4\ran & \text{ when }l=m. 
\end{cases}
\]
In either case, by the argument in the preceding case, we may suppose that the binary component is neither odd nor even, and we may further assume that $k\leq 4$ and $m\leq 8$.  Therefore, 
\[
m_2(L)=\begin{cases}
2^\frac{3m-14}{2} & \text{ when }k=l\\
2^\frac{4m+k-14}{2} & \text{ when }l=m
\end{cases}
\] 
and 
\[
m(L)>\pi^{-4}\cdot 2\cdot m_2(L)\cdot (1-2^{-4})\cdot \zeta(4)=\frac{m_2(L)}{2^4\cdot 3}
\]
and hence $m(L)>2$ in all but the exceptional case when $k=l=4$ and $m=8$.  In this exceptional case, we have 
\[
L_2\cong \lan \epsilon_1\ran\perp 2^4\lan\epsilon_2, \epsilon_3\ran\perp 2^8\lan \epsilon_4\ran.   
\]
This means
\[
\left(\mu_2L\right)_2\cong \lan \epsilon_1\ran\perp 2^2\lan\epsilon_2, \epsilon_3\ran\perp 2^6\lan \epsilon_4\ran.
\]
Then $\mu_2L$ has class number 1 by \cite[Theorem 3.14]{EH75}, and consequently must correspond to a lattice in \cite{KL} with 2-profile $(0,2,2,6)_2$.  There is only one such lattice in \cite{KL} and it has the local 2-adic splitting 
\[
 \lan 3\ran \perp 2^2\lan 3,7\ran \perp 2^6\lan 7\ran,
\]
and hence 
\[
L_2\cong  \lan 3\ran \perp 2^4\lan 3,7\ran \perp 2^8\lan 7\ran.
\]
From here we may conclude from \cite[Theorem 3.14]{EH75} that 
\[
\theta(O^+(L_2))=\{c\in \Q_2^\times:(c,-5)=1\}=\{1,5,6,14\}{\Q_2^\times}^2.
\]
Now we can use the formula given in \cite[102:7]{OM} to count the number of proper spinor genera in the genus of $L_2$, namely, 
\[
g^+(L)=\left[J_\Q:\Q^\times\prod_{p}\theta(O^+(L_p))\right],
\] 
where $J_\Q$ denotes the group of rational id\`eles.  For an arbitrary element ${\bf x}=(x_2,x_3,x_5,...)\in J_\Q$, we will show that ${\bf x}$ is in $\Q^\times\prod_{p}\theta(O^+(L_p))$.  Multiplying ${\bf x}$ by a suitably chosen scalar $a$, we know that $ax_p$ is a unit at every $p\in S$.  If $ax_2$ is either $1$ or $5$, then $a{\bf x} \in \Q^\times\prod_{p}\theta(O^+(L_p))$.  On the other hand, if $ax_2$ is either $3$ or $7$, then $2a{\bf x} \in \Q^\times\prod_{p}\theta(O^+(L_p))$.  Therefore, we may conclude that there is only one proper spinor genus in the genus of $L$, and since $g(L)\leq g^+(L)$, we conclude that $g(L)=1$, and consequently $h_s(L)=1$ implies $h(L)=1$ for such a form. 

Case IV: $0<k<l<m$. So 
\[  
L_2\cong \lan \epsilon_1\ran\perp 2^k\lan \epsilon_2\ran \perp 2^l\lan \epsilon_3\ran\perp 2^m\lan \epsilon_4\ran
\]
where $\epsilon_i\in \Z_2^\times$; we will consider the cases when $k$ is odd or even separately.  

First, we suppose that $k$ is even, so $k=2k'$ for some natural number $k'$, and define $L'=\mu_2^{k'}(L)$.  Then, 
\[
L'_2\cong \lan \epsilon_1, \epsilon_2\ran \perp 2^{l-k}\lan \epsilon_3\ran\perp 2^{m-k}\lan \epsilon_4\ran,
\]
where $\lan \epsilon_1, \epsilon_2\ran$ is a proper binary modular component.  If $l-k=1$, then $L'$ has 2-profile $(0,0,1,m-k)_2$, hence (\ref{spnnorm}) holds for $p=2$ by \cite[Theorem 3.14]{EH75}, implying that $h(L')=1$.  Therefore $L'$ is among the lattices in \cite{KL}, and must have a 2-profile from among 
\[
\begin{array}{ccc}
(0,0,1,2)_2 & (0,0,1,3)_2 & (0,0,1,4)_2,
\end{array}
\]
all of which in turn would make $L$ of type E, meaning that $h(L)=1$.  On the other hand, if $l-k=2$, then $L'$ has 2-profile $(0,0,2,m-k)_2$, and again $h(L')=1$ by \cite[Theorem 3.14]{EH75}.  Therefore $L'$ is among the lattices in \cite{KL}, and must have a profile from among 
\[
\begin{array}{ccc}
(0,0,2,3)_2  & (0,0,2,5)_2 & (0,0,2,4)_2.
\end{array}
\]
For the first two 2-profiles this once again forces $L_2$ to be of type E, implying $h(L)=1$. On the other hand, when $L'$ has 2-profile $(0,0,2,4)_2$, it is possible that $L'$ lifts either to a lattice with 2-profile $(0,2,4,6)_2$ or to a lattice with 2-profile $(0,0,4,6)_2$.  Using the algorithm from \S \ref{algorithm}, we generate all possible genera bearing such 2-profiles, and a check in Magma reveals that all of the associated spinor genera split into multiple classes.  Therefore in this case, we are assured that $h_s(L')>1$.  Cases where $l-k\geq 3$ can be reduced to one of these two cases above by repeated applications of $\mu_2$ to $L'$. 

Suppose that $k$ is odd, so $k=2k'+1$ for some natural number $k'$, and as above, define $L'=\mu_2^{k'}(L)$.  Then, 
\[
L'_2\cong \lan \epsilon_1\ran \perp 2\lan  \epsilon_2\ran \perp 2^{l-k+1}\lan \epsilon_3\ran\perp 2^{m-k+1}\lan \epsilon_4\ran.
\]
Now we will consider the possibility that $l-k$ is either odd or even.  If $l-k$ is odd, then $l-k+1=2\ell'$ for some natural number $\ell'$, and letting $L''=\mu_2^{\ell'-1}(L')$, we have 
\[
L''_2\cong \lan \epsilon_1\ran \perp 2\lan  \epsilon_2\ran\perp 2^2\lan \epsilon_3\ran \perp 2^{m-l+2}\lan \epsilon_4\ran.
\]
But then $L''_2$ is of type $E$ and hence $h(L')=1$ and has 2-profile $(0,1,2,m-l+2)_2$ where $m-l+2>2$. As no such profile exists among the lattices in \cite{KL}, this case cannot occur.  On the other hand, suppose that $l-k$ is even, so $l-k=2\ell'$ for some natural number $\ell'$, and define $L''=\mu_2^{\ell'-1}(L')$.  Then, 
\[
L'_2\cong \lan \epsilon_1\ran \perp 2\lan  \epsilon_2\ran \perp 2^{2\ell'+1}\lan \epsilon_3\ran\perp 2^{m-k}\lan \epsilon_4\ran,
\]
and hence 
\[
L''_2\cong \lan \epsilon_1\ran \perp 2\lan  \epsilon_2\ran \perp 2^{3}\lan \epsilon_3\ran\perp 2^{m-l+3}\lan \epsilon_4\ran.
\]
which is always of type E, and hence $L''$ has class number 1 and 2-profile $(0,1,3,m-l+3)_2$ where $m-l+3>3$, but no such lattice exists in \cite{KL}. 
\end{proof}


\begin{lem}\label{powersof3}
Only $3^6\in \mathcal D_s$. 
\end{lem}

\begin{proof}
Suppose that $L$ is a $\Z$-lattice with $h_s(L)=1$ and $h(L)>1$ for which $3\mid d(L)$.  Then in view of the previous lemmas we know that $d(L)=2^k\cdot 3^m$ where $k$ is a non-negative integer and $m$ is a positive integer.  By Proposition \ref{spngen} we know that $h(\muu L)=1$, and consequently $\muu L$ appears among the lattices in \cite{KL}, and must thus have one of the following 3-profile types:
\[
\begin{array}{cccccc}
(0,0,0,1)_3 & (0,0,1,1)_3 & (0,1,1,1)_3 & (0,0,0,2)_3 & (0,0,2,2)_3 & (0,2,2,2)_3.
\end{array}
\]
From the proof of Proposition \ref{spngen}, we know that $\Z_p^\times\subseteq\theta(O^+((\muu L)_p))$ at every $p\in S$. Define $L'$ to be the lattice for which $\mu_3(L')=\muu L$.  Then $L_p'\cong (\muu L)_p$ at every prime $p\neq 3$, and $L'$ must have 3-profile from among 
\[
\begin{array}{rrrrrr}
(0,0,0,3)_3 & *(0,0,1,3)_3 & (0,1,1,3)_3 & (0,0,0,4)_3 & (0,0,4,4)_3 & (0,4,4,4)_3\\
*(0,0,1,2)_3 & *(0,1,1,2)_3 & (0,1,3,3)_3 & (0,0,2,4)_3 & (0,2,4,4)_3 & *(0,1,2,2)_3\\
*(0,1,2,3)_3 & (0,3,3,3)_3 & (0,2,2,4)_3 & (0,0,2,3)_3 & *(0,2,3,3)_3 & (0,2,2,3)_3\\
(0,0,3,3)_3.\\
\end{array}
\]
We can immediately rule out all but the starred cases, since these profiles would have to correspond to an $L'$ with $h(L')=1$, but no such profiles appear among the lattices in \cite{KL}. Suppose that $L'$ has one of the starred 3-profiles above. Then $L'$ has $h(L')=1$ (except in certain exceptional cases corresponding to $(0,1,2,3)_3$) and $h_s(L')=1$.  In these cases, we define $L''$ to be the lattice for which $\mu_3(L'')=L'$.   Here we observe again that for every prime $p\neq 3$ we have $L''_p\cong L'_p\cong \muu L_p$.  Then $L''$ has a 3-profile coming from among 
\[
\begin{array}{cccccc}
(0,0,1,4)_3 & (0,1,4,4)_3 & (0,0,1,5)_3 	& (0,1,1,4)_3	& *(0,1,4,5)_3	& (0,4,5,5)_3\\
(0,0,3,4)_3 & (0,3,4,4)_3 & (0,0,3,5)_3	& *(0,1,3,4)_3	& *(0,3,4,5)_3 & *(0,2,3,4)_3 \\
*(0,1,2,5)_3	& (0,3,3,4)_3 & *(0,2,1,4)_3 & 	*(0,2,3,5)_3.\\
\\
\end{array}
\]
Once again, all but the starred cases correspond to profiles that would force $h(L'')=1$, but no such 3-profiles appear in \cite{KL}, so these can be immediately eliminated.  The remaining profiles (including $(0,1,2,3)_3$) will be dealt with case by case. 

First, suppose that $L''$ has 3-profile $(0,2,3,4)_3$ or $(0,1,2,4)_3$.  In this case $\mu_3(L'')=L'$ has 3-profile $(0,0,1,2)_3$, and therefore $L'$ is in \cite{KL}.  By searching among the lattices in \cite{KL} with 3-profile $(0,0,1,2)_3$ we find, using Sagemath, that any such lattice has $m_2(L')=1/6$.  Consequently, $m_2(L'')=m_2(L)=1/6$.  From here, we can compute upper and lower bounds for the mass for $L''$.  Since 
\[
\frac{1}{(1-p^{-4})}<\frac{1}{(1-p^{-2})^2}
\] 
and 
\[
m_3(L'')=\frac{3^\frac{13}{2}}{2^4}
\]
we can underestimate $m(L'')$ by 
\[
m^{-}(L'')=\pi^{-4}\cdot 2\cdot \frac{1}{6}\cdot 2\cdot \frac{3^\frac{13}{2}}{2^4}\cdot (1-2^{-4})\cdot (1-3^{-4})\cdot \zeta(4)=\frac{3^{1/2}\cdot 5}{2^4}\approx 0.5412
\]
and overestimate $m(L'')$ by 
\[
m^+(L'')=\pi^{-4}\cdot 2\cdot \frac{1}{6}\cdot 2\cdot \frac{3^\frac{13}{2}}{2^4}\cdot (1-2^{-2})^2\cdot (1-3^{-2})^2\cdot \zeta(2)^2=\frac{3^{3/2}}{2^3}\approx 0.6495
\]
where $m^{-}(L'')<m(L'')<m^{+}(L'')$.  But since $g(L'')=2$, and $m_s(L'')=m(L'')/g(L'')$ this means that 
\[
0.2707<m_s(L'')<0.3248.
\]
Consequently, $m_s(L'')$ is not of the form $1/\mid O(L'')\mid$, and therefore $h_s(L'')>1$. 

Next, suppose that $L''$ has 3-profile $(0,1,2,5)_3$ or $(0,2,3,5)_3$.  Again, we know that $h(L'')>1$ since no such profiles appear in \cite{KL}.  On other other hand, we know that $\mu_3(L'')=L'$ does appear in \cite{KL}, and so by searching among the lattices in \cite{KL}, and using Sagemath, we determine that $m_2(L'')=m_2(L)=1/6$ or $1/18$.  Since 
\[
m_3(L'')=\frac{3^8}{2^4}
\]
for either profile, we obtain 
\[
m(L'')=\pi^{-4}\cdot 2\cdot m_2(L'')\cdot 2\cdot \frac{3^8}{2^4}\cdot (1-2^{-2})^2\cdot (1-3^{-2})^2\cdot \zeta(2)^2=m_2(L'')\cdot \frac{3^4}{2^4}.
\]  
Since $m_2(L')=m_2(L'')$, this implies  
\[
m_s(L'')=\frac{m(L'')}{g(L'')}=m_2(L')\cdot \frac{3^4}{2^5},
\]
but this will always have at least one power of 3 in the numerator, and hence is not of the form $1/\mid O(L'')\mid$.  Therefore, we may conclude that $h_s(L'')>1$.  The cases when $L''$ has 3-profiles $(0,1,3,4)_3, (0,1,4,5)_3$ and $(0,3,4,5)_3$ follow similarly, except in these cases 
\[
m(L'')=m_2(L')\cdot \begin{cases}
\frac{3^3}{4}& \text{ for 3-profile } (0,1,3,4)_3\\
\frac{3^5}{4} & \text{ for 3-profile } (0,1,4,5)_3\\
\frac{3^4}{4} & \text{ for 3-profile } (0,3,4,5)_3,
\end{cases}
\]
where the possibilities for $m_2(L')=m_2(L'')$ are 
\[
\frac{1}{2^2\cdot 3^2},\,\,\frac{1}{2\cdot 3^2},\,\, \frac{1}{2^2\cdot 3},\,\, \frac{1}{3^2},\,\, \frac{1}{2^2},\,\, \frac{1}{3}\text{ or }1.
\]
But again, in every case $m_s(L'')$ is left with a 3 in the numerator, and hence is not of the form $1/\mid O(L'')\mid $.  

Finally, we deal with the case where $L'$ has 3-profile $(0,1,2,3)_3$.  From \cite{N}, we know that there are 33 isometry classes of lattices with discriminant $3^6$, and of these, only 6 have 3-profile $(0,1,2,3)_3$, namely, 
\[
L_1\cong\left[\begin{array}{cccc}
2 & 0 & 0 & 1 \\
0 & 6 & 3 & 0 \\
0 & 3 & 6 & 0 \\
1 & 0 & 0 & 14
\end{array}\right]
L_2\cong\left[\begin{array}{cccc}
2 & 1 & 0 & 0 \\
1 & 2 & 0 & 0 \\
0 & 0 & 18 & 9 \\
0 & 0 & 9 & 18
\end{array}\right]
L_3\cong\left[\begin{array}{cccc}
6 & 3 & 3 & 3 \\
3 & 6 & 0 & 3 \\
3 & 0 & 8 & 4 \\
3 & 3 & 4 & 8
\end{array}\right]
\]
where $L_1,L_2$ and $L_3$ are the representative classes for a single genus, and 
\[
M_1\cong\left[\begin{array}{cccc}
4 & 1 & 1 & 2 \\
1 & 4 & 1 & 2 \\
1 & 1 & 4 & -1 \\
2 & 2 & -1 & 16
\end{array}\right]
M_2\cong\left[\begin{array}{cccc}
4 & 2 & 1 & 1 \\
2 & 4 & -1 & 2 \\
1 & -1 & 10 & 4 \\
1 & 2 & 4 & 10
\end{array}\right]
\]
\[
M_3\cong\left[\begin{array}{cccc}
2 & 1 & 1 & 1 \\
1 & 8 & -1 & 2 \\
1 & -1 & 8 & 2 \\
1 & 2 & 2 & 8
\end{array}\right],
\]
where $M_1,M_2$ and $M_3$ are representatives for three distinct genera, each with class number 1. Here $L_1$ corresponds to the quadratic form (\ref{example}) and $\gen(L_1)$ splits into two spinor genera, namely $\spn(L_1)$ and $\spn(L_2)=\spn(L_3)$, and $h_s(L_1)=1$ while $h_s(L_2)=h_s(L_3)=2$.  Consequently, any lattice $L$ which descends to $L_2$ or $L_3$ by a series of $\mu_p$-transformations will already have $h_s(L)>1$. On the other hand, it is still possible to have a lattice $L$ descend to $L_1, M_1, M_2$ or $M_3$ by a series of $\mu_p$-transformations, which has spinor class number 1.  If $L$ descends by $\mu_3$, then this would imply that there is a lattice with spinor class number 1 and 3-profile 
\[
\begin{array}{ccc}
(0,3,3,4)_3 & (0,1,4,5)_3 & (0,3,4,5)_3.
\end{array}
\]
All of these would lead to a contradiction, since $(0,3,3,4)_3$ would have class number 1 but does not appear in \cite{KL}, and $(0,1,3,5)_3$ and  $(0,3,4,5)_3$ have already been ruled out in the preceding paragraphs using the mass formula.  Therefore, the only possibility is that $L$ descends to one of $L_1, M_1, M_2$ or $M_3$ by a series of $\mu_2$-transformations.  If we can show that there is no lattice with spinor class number 1 and discriminant $2^k\cdot 3^6$ for $k=2,4,6$, then we are done.  Using the list of 33 isometry classes with discriminant $3^6$ to seed the algorithm in \S \ref{algorithm}, we can generate all possible isometry classes of lattices with discriminants $2^k\cdot 3^6$ for $k=2,4,6$.  Generating this list in Magma, we obtain 18 genera, 63 genera, and 135 genera corresponding to discriminants $2^2\cdot 3^6$, $2^4\cdot 3^6$ and $2^6\cdot 3^6$, respectively.  Narrowing this list down to only the genera which admit 3-profile $(0,1,2,3)_3$, there are 8 genera, 28 genera, and 60 genera corresponding to discriminants $2^2\cdot 3^6$, $2^4\cdot 3^6$ and $2^6\cdot 3^6$, respectively.  Among these, there is only one genus which has class number $1$ and is therefore in \cite{KL}, namely, 
\[
K_1\cong\left[\begin{array}{cccc}
4 & 2 & -1 & 0 \\
2 & 10 & 4 & 0 \\
-1 & 4 & 10 & 3 \\
0 & 0 & 3 & 12
\end{array}\right].
\]
which has local 2-adic structure $\mathbb H\perp 2\lan 1,7\ran$.  Consequently any lattice which descends to $K_1$ by a $\mu_2$-transformation must have 2-profile $(0,0,3,3)_2$, but we already know from the algorithm that all lattices of discriminant $2^6\cdot 3^6$ have spinor class number greater than 1. 
\end{proof}

The proof of Theorem \ref{maintheorem} now follows by combining the above results and the fact that Nipp's tables \cite{N} explicitly cover the discriminant 729. 

\section{Appendices}

\subsection{Sagemath and Magma computations}\label{algorithm}

Computations done in Sagemath were standard lattice computations using the built in functionality of Sagemath, for example for the computation of local splittings. However, at several points in the proofs of Lemmas \ref{2powers} and \ref{powersof3} we needed to generate representatives of all isometry classes of $\Z$-lattices of a given discriminant. As this is not a standard capability of the available software, it was necessary to develop a method for doing this in cases encountered there. For this purpose, we produced an algorithm based on \cite[Lemma 3]{EN91}, which adapts a method used by Pall \cite{P46}, that, in conjunction with some of the built-in functionality in Magma for testing local and global isometry, can be used to generate representatives of the isometry classes of lattices of discriminant $D=D'p^2$ from those of discriminant $D'$, where $p\in S$.  The code described in what follows is available at \url{http://github.com/annahaensch/SpinorClass1} 

The algorithm is seeded with a list of representatives of the isometry classes of lattices of discriminant $D'$ given in Nipp's table \cite{N}. In order to cross-reference between the language of quadratic lattices, which we have chosen to use here, and the classical language of quadratic forms adopted by Nipp, we need to specify our conventions regarding the correspondence between forms and lattices. First associate to a primitive quadratic form $f=\sum_{1\leq i < j \leq 4}f_{ij}x_ix_j$ with $f_{ij}\in\Z$ the matrix $F$ of second partial derivatives of $f$; so disc($f$)=det($F$). If $f_{ij}$ is odd for at least one $i\neq j$, then the $\Z$-lattice $L_f$ with $L_f\cong F$ is a primitive lattice with $\mathfrak n(L) = 2\mathfrak s(L) = 2\Z$ and $\text{d}(L_f)=\text{disc}(f)$. If $f_{ij}$ is even for all $i\neq j$, then the $\Z$-lattice $L_f$  with $L_f\cong \frac{1}{2}F$ is a primitive lattice with $\mathfrak n(L) = \mathfrak s(L) =\Z$ and $16\text{d}(L_f)=\text{disc}(f)$.  

The algorithm is explicitly codified in the Github repository for the case when $p=2$, but it can be done similarly when $p=3$.  For the sake of illustration, we begin with a discriminant $D'=16$ and seed the algorithm with the set of matrices, {\tt Amatrices}, associated to the two distinct equivalence classes of forms of discriminant 16 (cf.  \cite[p. 23]{N}), and the set of 15 generating matrices, {\tt Pmatrices}, as described in \cite[Lemma 3]{EN91}.  In general, on the $k^\text{th}$ iteration, the algorithm will generate a list of quaternary lattice genera with associated discriminant $2^{2k}\cdot D'$.  On all but the last iteration it will also generate a list of class representatives.  Since this becomes a costly calculation as the discriminant increases, it is omitted from the last step.  After running to completion, {\tt Genera} will be a list of lists, wherein {\tt Genera[k]} is a complete list of quaternary lattice genera with, in this instance, discriminant $2^{2k+4}$, without redundancy.

\subsection{Table of ternary one-class spinor genera}
For the sake of completeness, we provide in Table 1 a list of representatives of all one-class spinor genera of ternary forms which are not in one-class genera, along with their discriminants.  These forms are by necessity spinor regular, and in some cases are also regular. The starred forms in the table are those which are regular.  The sextuple $[a,b,c,d,e,g]$ corresponds to the ternary form $f(x,y,z)=ax^2+by^2+cz^2+dyz+exz+gxy$ and the discriminant given is $disc(f)$.

\begin{table}[!htb]
\begin{tabular}{rl  c rl}
\hline
& & & \\
*54:& [1, 1, 9, 0, 0, 1 ] & &*2592:& [3, 4, 28, 4, 0, 0 ]\\
*54:& [1, 3, 3, 3, 0, 0 ] & &2744:& [7, 8, 9, 6, 7, 0 ]\\
*128:&[1, 1, 16, 0, 0, 0 ] & &*3456:& [1, 12, 36, 0, 0, 0 ]\\
128:& [2, 2, 5, 2, 2, 0 ] & &3456:& [4, 9, 12, 0, 0, 0 ]\\
*162:& [1, 3, 7, 0, 1, 0 ] & &*4096:& [1, 8, 64, 0, 0, 0 ]\\
&&\\
*216:& [1, 1, 36, 0, 0, 1 ] & &4096:& [4, 8, 17, 0, 4, 0 ]\\
*216:& [1, 3, 10, 3, 1, 0 ] & &7776:& [4, 9, 28, 0, 4, 0 ]\\
216:& [3, 3, 4, 0, 0, 3 ] & &8192:& [4, 9, 32, 0, 0, 4 ]\\
216:& [3, 4, 4, 4, 3, 3 ] & &8192:& [5, 13, 16, 0, 0, 2 ]\\
256:& [1, 4, 9, 4, 0, 0 ] & &8192:& [9, 9, 16, 8, 8, 2 ]\\
&&\\
*486:& [1, 7, 9, 0, 0, 1 ] & &10976:& [8, 9, 25, 2, 4, 8 ]\\
*512:& [1, 4, 16, 0, 0, 0 ] & &*13824:& [1, 48, 48, 48, 0, 0 ]\\
512:& [2, 5, 8, 4, 0, 2 ] & &*13824:& [4, 13, 37, 2, 4, 4 ]\\
512:& [4, 4, 5, 0, 4, 0 ] & &13824:& [9, 16, 16, 16, 0, 0 ]\\
648:& [1, 7, 12, 0, 0, 1 ] & &13824:& [13, 13, 16, -8, 8, 10 ]\\
&&\\
686:& [2, 7, 8, 7, 1, 0 ] & &32768:& [9, 16, 36, 16, 4, 8 ]\\
*864:& [1, 3, 36, 0, 0, 0 ] & &32768:& [9, 17, 32, -8, 8, 6 ]\\
*864:& [1, 12, 12, 12, 0, 0 ] & &*41472:& [3, 16, 112, 16, 0, 0 ]\\
864:& [3, 4, 9, 0, 0, 0 ] & &124416:& [9, 16, 112, 16, 0, 0 ]\\
864:& [4, 4, 9, 0, 0, 4 ] & &175616:& [29, 32, 36, 32, 12, 24 ]\\
&&\\
*1944:& [1, 7, 36, 0, 0, 1 ] & &\\
*2048:& [1, 16, 16, 0, 0, 0 ] & &\\
2048:& [4, 5, 13, 2, 0, 0 ] & &\\
2048:& [4, 9, 9, 2, 4, 4 ] & &\\
2048:& [5, 8, 8, 0, 4, 4 ] & &\\
&&\\
       \hline
\end{tabular}
\caption{Complete list of primitive positive definite ternary quadratic forms in one class spinor genera, but not one-class genera, listed with their discriminants.  An asterisk before the entry indicates that the form is regular.}\label{Tab:1}
\end{table}

\subsection*{Acknowledgements}
The second author wishes to thank the Max Planck Institute for Mathematics for the hospitality and support provided throughout her visit during which much of this research took place.

\end{document}